\documentclass[11pt,reqno]{amsart}
\usepackage{amsfonts}

\usepackage{amssymb, amsmath}
\usepackage{xcolor}
\usepackage{lscape}
\usepackage{dsfont}
\usepackage{mathrsfs}
\usepackage{bbm}
\usepackage[pdftex,plainpages=false,colorlinks,hyperindex,bookmarksopen,linkcolor=red,citecolor=blue,urlcolor=blue]{hyperref}
\usepackage{graphicx}
\usepackage{subcaption}

\usepackage{verbatim}

\usepackage{geometry}
\DeclareMathAlphabet{\mathpzc}{OT1}{pzc}{m}{it}
\newtheorem{te}{Theorem}[section]
\newtheorem{defin}[te]{Definition}
\newtheorem{os}[te]{Remark}
\newtheorem{prop}[te]{Proposition}

\numberwithin{equation}{section}
\allowdisplaybreaks

\def \l { \left( }
\def \r {\right) }

\begin{document}

\title[]{  
para-Markov chains\\ and related non-local equations
}
\keywords{Fractional operators, Non-Markovian dynamics, Semi-Markov chains, Schur systems of lifetimes, Functions of matrices.
\\
* Dipartimento di Scienze Statistiche - Sapienza Università di Roma
\\
** Dipartimento di Matematica "Giuseppe Peano" - Università degli studi di Torino
\\
Corresponding author: Enrico Scalas
}
\date{\today }
\subjclass{60G55, 60G09, 60G99}

\thanks{The authors acknowledge financial support under the National Recovery and Resilience Plan (NRRP), Mission 4, Component 2, Investment 1.1, Call for tender No.
104 published on 2.2.2022 by the Italian Ministry of University and Research (MUR), funded by the
European Union - NextGenerationEU- Project Title “Non-Markovian Dynamics and Non-local Equations”
- 202277N5H9 - CUP: D53D23005670006 - Grant Assignment Decree No. 973 adopted on June 30, 2023,
by the Italian Ministry of Ministry of University and Research (MUR) }
\thanks{
The author Bruno Toaldo would like to thank the Isaac Newton Institute for Mathematical Sciences, Cambridge, for support and hospitality during the programme Stochastic Systems for Anomalous Diffusion, where work on this paper was undertaken. This work was supported by EPSRC grant EP/Z000580/1.}

 \author[]{Lorenzo Facciaroni*}
 \author[]{Costantino Ricciuti*}
 \author[]{Enrico Scalas*}
 \author[]{Bruno Toaldo**}
 
\begin{abstract}
There is a well established theory that links
   semi-Markov chains having Mittag-Leffler waiting times to time-fractional equations.
   We here go beyond the semi-Markov setting, by defining some non-Markovian chains whose waiting times, although marginally Mittag-Leffler,  are assumed to be stochastically dependent. This creates a long memory tail in the evolution, unlike what happens for semi-Markov processes.
As a special case of our chains, we study a particular counting process which extends the well-known fractional Poisson process, the last one having independent, Mittag-Leffler waiting times.

\end{abstract}

\maketitle

\tableofcontents
{}
{}

\section{Introduction}

Continuous-time semi-Markov chains are obtained from Markov chains by relaxing the assumption of exponential waiting times.  Several papers were devoted to chains with Mittag-Leffler distributed waiting times, among which the well-known fractional Poisson process is included (the reader can consult e.g. \cite{ascione2021, orsingherbeghin2009, degregorio2021, georgiu2015, kolokoltsov2009, laskin2003, mainardi2004, scalas2004, meerschaert2011,  toaldo2019, polito2010, ricciuti2023semi}). A Mittag-Leffler distribution is characterized by a parameter $\nu \in (0,1]$, such that if $\nu=1$  one re-obtains the exponential distribution and a Markov chain, while, for  $\nu \in (0,1)$,  the distribution has a density $\psi$  with asymptotic  power-law decay, $\psi (t)\sim t^{-\nu -1}$. In the latter case, the waiting times have infinite mean and variance, which is useful in models of anomalous diffusion (see e.g. \cite{metzler2000}) as well as in financial applications (see e.g. \cite{scalas2021}). 

In the Markovian case, the transition matrix $P(t)$ solves the Kolmogorov equation
$$\frac{d}{dt}P(t) = G P(t) $$
where $G$ is the infinitesimal generator of the chain. For semi-Markov chains with Mittag-Leffler waiting times, this equation is replaced by  
$$\frac{d^\nu}{dt^\nu}P(t) = G P(t) $$
where $d^\nu/dt^\nu$ denotes the Caputo fractional derivative, see also \cite{garra2015, kataria2019, orsingher2018semi, ricciuti2017semi} for the case where $\nu$ depends on the current state of the process, which is applied to models of anomalous diffusion in heterogenerous media.

Note that in the Markovian case, corresponding to $\nu=1$, one has a ``local''
 equation:  the time derivative - a local operator -  is  consistent with the lack-of-memory property typical of Markov chains. Indeed, the connection between Markov processes and first order differential equation is well-established.

On the contrary, if $\nu \in (0,1)$ one has a ``non-local'' equation:  the integral operator of Volterra type $d^\nu/dt^\nu $ contains a memory kernel.

A limitation of these models is as follows.
Semi-Markov processes lose memory of the past at renewal times only, i.e. when the chain jumps. Therefore, the future evolution of a semi-Markov process is only influenced by the recent past, and not by the whole past history. In this sense, the time-fractional derivative only contains information on the recent past through the age (i.e. the time elapsed from the previous renewal event).

In general, the study of non-Markovian processes is difficult. Therefore, any progress in this direction is of potential interest to a wide community of mathematicians and applied scientists. One goal is then to use the apparatus of fractional operators and time-fractional equations to treat the long-memory tail of some non-Markovian processes (which are not semi-Markov).

Our contribution to this goal is to define a class of processes that we call \textit{(fractional) para-Markov chains}. These processes have a property in common with semi-Markov ones: the marginal distribution of each waiting time is Mittag-Leffler. However, all the waiting times are stochastically dependent, hence the process  keeps memory of the whole past history. We then extend the mathematical techniques typically used for semi-Markov chains, including the use of fractional operators, to such a class of non-Markovian processes. Eventually, we obtain a governing equation of the form
$$\frac{d^\nu}{dt^\nu}P(t) = -(-G)^\nu P(t). $$
What makes these processes analytically tractable is that, in distribution, they are proved to be a suitable mixture of Markov processes, hence the choice of the name {\em para-Markov}.

Our most general results concern the case of finite state space. However,  we also deal with an important case with infinite countable state space, namely an extension of the fractional Poisson process. We refer to it as the \textit{exchangeble fractional Poisson process}, for reasons that will become clear later. This is a counting process with unit jumps whose waiting times, although marginally having a Mittag-Leffler distribution, present an appropriate stochastic dependence. The latter is given by a particular Schur-type distribution (in the sense of \cite{barlow1992, caramellino1994dependence, caramellino1996wbf}). Therefore, unlike the fractional Poisson process, our counting process is not a renewal process.

The structure of the paper is the following: section 2 is devoted to some preliminaries on Markov and semi-Markov chains; section 3 deals with the exchangeable fractional Poisson process; in section 4 we introduce the general theory of para-Markov chains and study the finite-state case.

\section{Preliminaries}\label{preliminaries}

Let us consider a sequence of non-negative random variables $\theta = \{\theta_n\}_{n=1}^{\infty}$, which we interpret as the sequence of waiting times, and the stochastic process $T := \{T_n,\ n\in\mathbb{N}\}$ such that
\begin{align*}
    T_n := \sum_{k = 1}^n \theta_k,
\end{align*}
with the convention $T_0 := 0$.
Let $\mathcal{S}$ be a countable state space and let $Y = \{Y_n,\ n\in\mathbb{N}\}$ be a discrete-time stochastic process which takes value in $\mathcal{S}$. We say that the process $X = \{ X_t,\ t\geq0 \}$, defined by
\begin{align*}
    X_t = Y_n \qquad \qquad t\in[T_n, T_{n+1}),\ n\in\mathbb{N},
\end{align*}
is a continuous-time chain.

We will consider three types of continuous-time chains, say Markov, semi-Markov, and para-Markov, the last one being introduced in this paper. In all three cases, the embedded chain $Y$ is a discrete-time Markov chain, and thus what distinguishes one from the others is the joint distribution of the waiting times.
Let $H = [h_{ik}]$ be the transition matrix of $Y$, such that
\begin{align*}
    h_{ik} := \mathbb{P}[ Y_{n+1}=k | Y_n=i ]  \qquad i,k \in \mathcal{S}
\end{align*}
under the convention $h_{ii}=0$.

Consider $\lambda : \mathcal{S} \to (0,\infty)$. The process $X$ is a \textit{continuous-time Markov chain} if (consult   \cite{norris1998}, page 94) the waiting times are such that
\begin{align}
    \mathbb{P}[\theta_1 > t_1, \ldots, \theta_n > t_n | Y_0 = y_0,\ldots, Y_{n-1} = y_{n-1}] &= e^{-\lambda(y_0)t_1}\cdots e^{-\lambda(y_{n-1}) t_{n}},  \label{Preliminaries: conditional dependence structure of continuous-time Markov chains}
\end{align}
i.e. the $\theta_i$s are conditionally independent, each of them having exponential distribution
\begin{align}
    \mathbb{P}[\theta_i > t | Y_{i-1} = x] = e^{-\lambda(x) t}.\label{Preliminaries: exponential distribution for waiting times}
\end{align}
According to the above definition, the Markov chain is homogeneous in time. 

A key object of continuous-time Markov chains is the generator of the process, denoted as $G = [g_{ij}],\ i,j\in S$. If $i\neq j$, then $g_{ij}$ represents the instantaneous rate at which the process moves from state $i$ to state $j$, i.e. $g_{ij}=\lambda (i) h_{ij}$.
Moreover, from  \eqref{Preliminaries: exponential distribution for waiting times}, the time spent in state $i$ before transitioning to another state, is exponentially distributed with rate $-g_{ii} =\lambda(i)$. In compact form, we can write the generator as 
\begin{align}g_{ij}= \lambda (i) (h_{ij}-\delta _{ij}). \label{generatore markoviano component-wise}\end{align}

The probability that the process is in state $j$ at time $t$, given that it started in $i$ at time 0, is denoted by $P_{ij}(t)$. We shall use the matrix form with $P(t) = [P_{ij}(t)]$. The family of operators $\{P(t),\ t\geq 0\}$ satisfies the semi-group property $P(t+s) = P(t)P(s)$ and it is the solution of the system of Kolmogorov backward $(a)$ and forward $(b)$ equation
\begin{align} \label{kolmogorov equations markov}
    (a)\begin{cases}
        \frac{d}{dt}P(t) = G P(t)
        \\ P(0) = I
    \end{cases}
    \qquad \qquad (b)\begin{cases}
        \frac{d}{dt}P(t) = P(t) G
        \\ P(0) = I
    \end{cases}.
\end{align}

\textit{Semi-Markov chains} are obtained from Markov ones by relaxing the assumption of exponential waiting times. Then equation \eqref{Preliminaries: exponential distribution for waiting times} is replaced by
\begin{align}
    \mathbb{P}[\theta_n > t | Y_{n-1} = x] = S_x(t), \label{Preliminaries: generic distribution for waiting times}
\end{align}
where $S_x(\cdot)$ is a generic survival function; this implies that the lack-of-memory property is satisfied only at time instants when jumps occur.

Let the distribution \eqref{Preliminaries: generic distribution for waiting times} be absolutely continuous with density $f_x(\cdot)$.
Moreover, let $p_{ij}(t)$ be the probability that the process $X$ moves from state $i$ to state $j$ in a time interval $[0,t]$, under the condition that $0$ is a renewal time, and let $P=[p_{ij}(t)]$ be the transition matrix. The family $\{P(t), t\geq 0\}$ cannot satisfy the semigroup property, i.e. $P(t+s)\neq P(t)P(s)$ unless in the Markovian case where $S_x$ is exponential.
By standard conditioning arguments, one can see that
\begin{align}
    p_{ij}(t) = \sum_{k\in S} h_{ik} \int_0^t f_i(\tau)\, p_{kj}(t-\tau) d\tau + S_i(t)\, \delta_{ij}. \label{Preliminaries: Markov Renewal equation}
\end{align}
Equation \eqref{Preliminaries: Markov Renewal equation} is called the semi-Markov renewal equation.

We are interested in semi-Markov chains whose waiting times follow the so-called  Mittag-Leffler distribution, which is defined by a particular choice of the survival function  \ref{Preliminaries: generic distribution for waiting times}.
\begin{defin}\label{Definition of Mittag-Leffler ranom variable}
    A non-negative random variable $J$ is said to follow   a Mittag-Leffler distribution with parameters $\nu\in(0,1]$ and $\lambda \in (0, \infty)$ if 
    \begin{align*}
        \mathbb{P}(J>t)= \mathcal{M}_{\nu}(-\lambda t^{\nu}), \;\; t \geq 0
    \end{align*}
   where $\mathcal{M}_{\nu} (\cdot)$ is the one-parameter Mittag-Leffler function, defined by
   \begin{align} \mathcal{M}_\nu (z) = \sum _{k=0}^\infty \frac{z^k}{\Gamma (1+\nu k)} \qquad z\in \mathbb{C}.\label{funzione Mittag Leffler}\end{align}
\end{defin}

Definition (\ref{Definition of Mittag-Leffler ranom variable}) gives an absolutely continuous distribution. So, consider semi-Markov chains whose waiting times are such that
\begin{align}
        \mathbb{P}[\theta_k > t | Y_{k-1} = x] = \mathcal{M}_{\nu}(-\lambda(x)t^{\nu}), \qquad x\in\mathcal{S},\label{Preliminaries: ML distribution for waiting times}
\end{align}
i.e., conditionally to $Y_{k-1}=x$, the variable $\theta _k$ has Mittag-Leffler distribution with parameters $\nu$ and $\lambda (x)$.
For $\nu=1$, \eqref{Preliminaries: ML distribution for waiting times} reduces to an exponential distribution and hence the process becomes a continuous-time Markov chain, with generator $G$ defined by $g_{ij}= \lambda (i) (h_{ij}-\delta _{ij})$. For $\nu \in (0,1)$ the process $X$ is semi-Markov.

Moreover, it is known that such a process is governed by the following backward $(a)$ and forward $(b)$ fractional equations (for a proof sketch, which is based on the renewal equation \eqref{Preliminaries: Markov Renewal equation}, see Proposition 2.1 in \cite{ricciuti2017semi} and references therein):    

\begin{align}
    (a)\begin{cases}
        \frac{d^\nu}{dt^\nu}P(t) = G P(t)
        \\ P(0) = I
    \end{cases}
    \qquad \qquad (b)\begin{cases}
        \frac{d^\nu}{dt^\nu}P(t) = P(t) G
        \\ P(0) = I
    \end{cases}. \label{Kolmogorov backward and forward}
\end{align}
Note that the state space is not required to be finite. As a particular case, the fractional Poisson process is obtained by setting $\lambda (x) =\lambda$ for any $x\in \mathcal{S}$ and $h_{i,j}=1$ if  $j = i+1$ and $h_{ij} = 0$ otherwise, whence $G$ is such that $g_{ii}= -\lambda$     and $g_{i,i+1}=\lambda$. Therefore, considering the forward system \textit{(b)} and setting $p_{0j}(t):=p_j(t)$, we obtain the governing equation often reported in the literature:
\begin{align}\frac{d^\nu}{dt^\nu}p_j(t)= -\lambda p_j(t)+\lambda p_{j-1}(t) \qquad p_j(0)= \delta_{j,0} \label{equazione fractional poisson}.\end{align}
There are many extensions of the fractional Poisson process, such as the time inhomogeneous extensions defined in \cite{beghin2019, leonenko2017}, as well as the counting processes studied in \cite{dicrescenzo2016, gupta2023, maheshwari2019}.

\section{The Exchangeable fractional Poisson process}

The well-known fractional Poisson process has independent, Mittag-Leffler waiting times between arrivals. 
The goal here is to build a counting process which, in analogy to the fractional Poisson process, increases by 1 unit when an event occurs and each waiting time has marginal Mittag-Leffler distribution. However,  we relax the hypothesis of independence between waiting times, as shown in the following definition.

\begin{defin} \label{Definition of exchangeable fractional Poisson process}

Let $\{J_k\}_{k=1}^\infty$ be a sequence of non-negative random variables, such that, for all $n\in \mathbb{N}\setminus \{0\}$ we have
\begin{align}
        \mathbb{P}[J_1 > t_1,\ \ldots,\ J_n > t_n] = \mathcal{M}_{\nu}\left(-\lambda^{\nu} \left(\sum_{k=1}^n t_k\right)^{\nu}\right) \qquad \qquad \nu\in(0,1],\ \ \lambda \in (0,\infty), \label{exchangeable fractional Poisson: joint survival function for (J_1, ..., J_n)}
    \end{align}
 where $t_k\geq0,\ k \in\{1,\ldots,n\}$. Moreover let $T_n:= \sum_{k=1}^nJ_k$, with the convention $T_0 := 0$.
Then the process $N = \{N_t,\ t\geq0\}$ defined by
    \begin{align*}
        N_t = n \qquad \qquad t\in[T_n, T_{n+1})
    \end{align*}
    is said to be \textbf{exchangeable fractional Poisson process} with parameters $\lambda$ and $\nu$.
\end{defin}
 We note that each $J_k$ follows a marginal Mittag-Leffler distribution with parameters $\lambda ^\nu$ and $\nu$, in the sense of Definition \ref{Definition of Mittag-Leffler ranom variable}; this can be obtained from formula \eqref{exchangeable fractional Poisson: joint survival function for (J_1, ..., J_n)}   with  $t_j = 0$ for each $j\neq k$. Another important feature is that the above sequence of waiting times is an infinite \textit{Schur-constant} sequence. We recall that a sequence $\{X_k\}_{k=1}^\infty$ of non-negative random variables is said to be an infinite \textit{Schur-constant} sequence if, for any $n\in \mathbb{N}\setminus \{0\}$, we have $\mathbb{P}(X_1>t_1, X_2>t_2, \dots, X_n>t_n)= S(t_1+t_2+\dots + t_n)$, for a suitable function $S$ which does not depend on $n$. This is a particular model of \textit{exchangeable} waiting times, in the sense that $S$ depends on the $t_k$ through their sum only, whence the name we have chosen for our counting process; this feature makes the process easily tractable from a statistical point of view and has many applications, see \cite{barlow1992, caramellino1994dependence, caramellino1996wbf}.

We further observe that for $\nu = 1$ we have that $\mathcal{M}_{1}(x)= e^x$ and \eqref{exchangeable fractional Poisson: joint survival function for (J_1, ..., J_n)} has the form
\begin{align}
        \mathbb{P}[J_1 > t_1,\ \ldots,\ J_n > t_n] = e^{- \lambda \sum _{k=1}^n t_k},
        \end{align}
namely the waiting times are i.i.d. exponential and $N_t$ is a Poisson process of parameter $\lambda$.

\begin{os}
From the joint survival function   \eqref{exchangeable fractional Poisson: joint survival function for (J_1, ..., J_n)}  it is possible to obtain the joint distribution function.   Indeed, by observing that
\begin{align*}
    \left\{ J_1 \leq t_1, \ldots, J_n \leq t_n \right\}^{c} = \left\{\{J_1 > t_1\}\cup \cdots \cup \{J_n > t_n\} \right\}, 
\end{align*}
by Poincar\'e Theorem we have 
\begin{align}
    & \mathbb{P}[J_1 \leq t_1, \ldots, J_n \leq t_n] \notag \\ 
    & = 1 - \sum_{i } \mathbb{P}[J_i > t_i] + \sum_{ i< j} \mathbb{P}[J_i > t_i, J_j > t_j] +  \cdots + 
     (-1)^n  \mathbb{P}[J_{1} > t_1, \ldots, J_{n} > t_n] \notag \\
&= 1 - \sum_{i }    \mathcal{M}_{\nu} (-\lambda ^\nu t_i^\nu)               + \sum_{ i< j}   \mathcal{M}_{\nu}\left(-\lambda^{\nu} \left(t_i+t_j\right)^{\nu}\right)   + \dots         +  (-1)^n
     \mathcal{M}_{\nu}\left(-\lambda^{\nu} \left(\sum_{k=1}^n t_k\right)^{\nu}\right). \label{cdf waiting times}
\end{align}
Moreover, we also get the joint density as
    \begin{align}
        f(t_1\ldots, t_n) = (-1)^n \frac{\partial^n}{\partial t_1 \cdots \partial t_n} \mathcal{M}_{\nu}\left(-\lambda^{\nu} \left(\sum_{k=1}^n t_k\right)^{\nu}\right). \label{densita waiting times}
    \end{align}
\end{os}
\begin{os}
By using $f(t_1, \dots, t_n)$, we can obtain the density of  $T_n$, the time of the $n$-th jump:
    \begin{align}
       f_{T_n}(u)&= \frac{d}{du} \mathbb{P}[T_n \leq u ] \notag \\ &= \frac{d}{du} \mathbb{P}\left[ \sum_{k = 1}^n J_k \leq u \right] \notag
        \\ &= \frac{d}{du} \int _{t_1+t_2+ \dots t_n \leq u} f(t_1, \dots, t_n) dt_1\dots dt_n \notag \\
        &= \frac{(-1)^n}{\Gamma (n)} u^{n-1} \mathcal{M}^{(n)}_{\nu} (-\lambda ^\nu u^\nu) \qquad u>0. \label{densita Tn}
    \end{align}
    This can be seen as a generalization of the so-called Erlang distribution that is recovered for $\nu=1$.
\end{os}

Before discussing some properties of the exchangeable fractional Poisson process, we need to recall the following definition (see \cite{Lamperti_James, lamperti1958, orsingher2013}).
\begin{defin} \label{Definition of Lamperti distribution}
    A non-negative random variable $L$ follows a Lamperti distribution of parameter $\nu\in(0,1]$ if its Laplace transform is given by
\begin{align}
    \mathbb{E}\left[e^{-\eta L}\right] = \mathcal{M}_{\nu}\left(-\eta^{\nu}\right), \qquad \eta \geq 0. \label{lamperti laplace}
\end{align}
\end{defin}
\begin{os}\ 
    \begin{enumerate}
        \item For $\nu = 1$ we get $\mathbb{E}\left[e^{-\eta L}\right] = e^{-\eta}$ which implies $L = 1$ almost surely.
        \item For $\nu\in(0,1)$ then $L$ is absolutely continuous with density given by
        \begin{align}
            f(t) = \frac{\sin(\pi \nu)}{\pi}\frac{t^{\nu - 1}}{t^{2\nu} + 2t^{\nu} \cos(\pi\nu) + 1} \ \ \ \ \ \ t > 0. \label{Lamperti density}
        \end{align}
    \end{enumerate}    
\end{os}

The following theorem shows that the exchangeable fractional Poisson process is equal in distribution to a time-changed Poisson process. This time-change consists in a random scaling of time based on a Lamperti variable.

\begin{te} \label{Theorem: time-change for exchangeable fractional Poisson process}
    Let $Q = \{Q_t,\ t\geq0\}$ be a Poisson process with intensity $\lambda$ and $N = \{N_t,t\geq0\}$ be the exchangeable fractional Poisson process, with parameters $\lambda$ and $\nu$.
    Let $L$ have Lamperti distribution with parameter $\nu$. 
    Then we have
    \begin{align*}
        N_t \overset{d}{=} Q_{Lt},\quad \forall t \geq0,
    \end{align*}
    where $\overset{d}{=}$ denotes equality of  finite dimensional distributions.
\end{te}
In section \ref{para-Markov section} below, we will provide the proof of Theorem \ref{Proposition:  para-Makov chains as time-change of Markov chains through Lamperti}, which includes Theorem \ref{Theorem: time-change for exchangeable fractional Poisson process} as a particular case. For this reason, here we omit the proof of Theorem \ref{Theorem: time-change for exchangeable fractional Poisson process}.    \\

Once again, we stress that for $\nu = 1$ we have $L = 1$ almost surely, that is the time parameter $Lt = t$ is no longer stochastic, obtaining the Poisson case as a special case of exchangeable fractional Poisson process.\\

The equivalence in distribution of Theorem \ref{Theorem: time-change for exchangeable fractional Poisson process} leads to the governing equation of the process. We shall use the notation $\mathbb{P}[N_t = k ] =: p_k(t)$.
\begin{te}
    Let $N$ be the exchangeable fractional Poisson process defined in \ref{Definition of exchangeable fractional Poisson process}. Then 
    \begin{align}
        \frac{d^\nu}{dt^\nu}p_k(t) = -\lambda^{\nu} (I-B)^{\nu} p_k(t) \qquad p_k(0)= \delta _{k,0}\label{Govening equation of the exchangeable fractional Poisson process}
    \end{align}
    where
    $B$ is the shift operator such that $Bp_k(t)=: p_{k-1}(t)$ and $$(I-B)^{\nu} p_k(t) = \sum_{j = 0}^{\infty} \left(\begin{matrix}
        \nu \\ j
    \end{matrix}\right) (-1)^j B^j p_k(t)=  \sum_{j = 0}^{\infty} \left(\begin{matrix}
        \nu \\ j
    \end{matrix}\right) (-1)^j  p_{k-j}(t)  $$ 
\end{te}
\begin{proof}
 Recalling that the Poisson process $Q_t$ is such that
    \begin{align*}
       \mathbb{E}\left[e^{-\eta Q_t}\right] = e^{-\lambda t ( 1 - e^{-\eta})} \qquad \eta \geq 0,
    \end{align*}
   and using  Theorem \ref{Theorem: time-change for exchangeable fractional Poisson process} we have that  $Q_{Lt}$ has the following moment generating function:
\begin{align}
A(\eta, t)&= \mathbb{E}\left[e^{-\eta Q_{Lt}}\right] \nonumber\\ 
          &= \int _0^\infty  e^{-\lambda t ( 1 - e^{-\eta})l} \mathbb{P}(L\in dl)\nonumber\\
          &= \mathbb{E}\left[ e^{-t\lambda(1-e^{-\eta}) L } \right]\nonumber
        \\ &= \mathcal{M}_{\nu}(-\lambda^{\nu}(1-e^{-\eta})^{\nu}t^{\nu}). \label{moment generating function of Q(Lt)}
    \end{align}
    Given that $t\mapsto \mathcal{M}_\nu(ct^{\nu})$ is an eigenfunction of the Caputo derivative $d^{\nu}/dt^{\nu}$ with eigenvalue $c$, we get
    \begin{align}
        \frac{d^{\nu}}{dt^{\nu}} A(\eta,t) = -\lambda^{\nu}\left(1-e^{-\eta}\right)^{\nu} A(\eta, t) \qquad A(\eta, 0) = 1. \label{Equazione funzione generatrice}
    \end{align}
    Using the equality 
    \begin{align*}
        (1 - e^{-\eta})^{\nu} = \sum_{j = 0}^{\infty} \left(\begin{matrix}
            \nu \\ j
        \end{matrix}\right) (-1)^j e^{-\eta j}
    \end{align*}
    and applying the inverse Laplace transform in $\eta$ on both sides of Equation \eqref{Equazione funzione generatrice} we get the thesis. Indeed, it is possible to prove that $d^\nu p_k(t)/dt^\nu$ is well posed, by using similar arguments as in point $(3)$ of the proof of Theorem \ref{Theorem: G irriducibile allora abbiamo 3 tesi. etc etc} : despite the state space is infinite,  the right hand side of equation (\ref{Govening equation of the exchangeable fractional Poisson process}) actually has a finite number of addends as in Theorem \ref{Theorem: G irriducibile allora abbiamo 3 tesi. etc etc}, because $$(I-B)^\nu p_k(t)=\sum_{j = 0}^{k} \left(\begin{matrix}
        \nu \\ j
    \end{matrix}\right) (-1)^j  p_{k-j}(t).$$

    \ \\
\end{proof}

Before giving a final result, we recall the formula by Faà di Bruno that generalizes the chain rule for the $n$-th derivative of a function composition, see \cite{faa1855}. For $f, u$ satisfying appropriate regularity conditions, we have
\begin{align}
    \l\frac{d}{dx}\r^{n} f(u(x))  = n! \sum_{k = 1}^n \frac{f^{(k)}(u(x))}{k!} \sum_{h_1+\cdots+h_k = n} \prod_{i = 1}^k\frac{u^{(h_i)}(x)}{h_i!} \label{Formula di Faa di Bruno per la derivazione di funzioni composte}
\end{align}
where the second sum is over all $k$-tuples of non-negative integers $(h_1,\ldots, h_k)$ satisfying the constraint $\sum_{i = 1}^kh_i = n.$

We also recall that the $n$-th derivative of the Mittag-Leffler function $\mathcal{M}_\nu(z)$ is 
given by \cite{garrappa2017, prabhakar1971}
\begin{equation}
\label{MLDerivative}
\l\frac{d}{dz}\r^{n} \mathcal{M}_\nu (z) = n! \mathcal{M}^{n+1}_{\nu,n \nu+1} (z),
\end{equation}
where $\mathcal{M}^\gamma_{\alpha,\beta} (z)$ is the so-called Prabhakar function or three-parameter Mittag-Leffler function defined as
\begin{equation}
\label{prabhakarfunction}
\mathcal{M}^{\gamma}_{\alpha,\beta} (z) = \frac{1}{\Gamma(\gamma)} \sum_{k=0}^\infty \frac{\Gamma(k+\gamma) z^k}{k! \Gamma(\alpha k + \beta)}.
\end{equation}

In the following theorem we find the explicit expression of $p_n(t)$ that solves the governing equation \eqref{Govening equation of the exchangeable fractional Poisson process}.
\begin{te}
     Let us consider the exchangeable fractional Poisson process $N$ defined in \ref{Definition of exchangeable fractional Poisson process}. Then the marginal distribution of $N_t$, $p_n(t) = \mathbb{P}(N_t =n)$, is given by
     \begin{align}
        p_n(t) = \begin{cases}
            \mathcal{M}_{\nu}(-\lambda^{\nu}t^{\nu}) & n= 0
            \\ \sum_{k = 1}^n (-1)^{(n+k)} (\lambda t)^{k \nu} \mathcal{M}_{\nu, k\nu + 1}^{k+1}(-\lambda^{\nu}t^{\nu}) c(k,n;\nu) & n\geq 1
        \end{cases} \label{Distribution of N(t) formule di Faa di Bruno}
    \end{align} 
    with 
    \begin{align*}
        c(k,n;\nu) =  \sum_{h_1 + \cdots + h_k = n} \prod_{i = 1}^k\frac{(\nu)_{h_i}}{h_i!} 
    \end{align*}
    where $(\nu)_h := \nu (\nu - 1)\cdots (\nu - h +1)$ and the sum is over all $k$-tuples of non-negative integers $(h_1,\ldots, h_k)$ satisfying the constraint $\sum_{i = 1}^kh_i = n.$
\end{te}
\begin{proof}
    Theorem \ref{Theorem: time-change for exchangeable fractional Poisson process} guarantees that $N$ is equal in distribution to a Poisson process with stochastic time parameter $L t$, where $L$ follows the Lamperti distribution of parameter $\nu$. Then, by conditioning, we have
    \begin{align}
        p_n(t) &= \int_0^{\infty} \frac{e^{-\lambda l t}}{n!}\left(\lambda l t\right)^n \mathbb{P}(L\in dl) \nonumber
        \\ &= \frac{t^n}{n!}(-1)^n\left(\frac{d}{dt}\right)^n\int_0^{\infty}e^{-\lambda lt} \mathbb{P}(L\in dl) \nonumber
        \\ &= \frac{t^n}{n!}(-1)^n\left(\frac{d}{dt}\right)^n \mathbb{E}\left[e^{-\lambda t L}\right]  \nonumber
        \\ &= \frac{t^n}{n!}(-1)^n\left(\frac{d}{dt}\right)^n \mathcal{M}_{\nu}\left(-\lambda^{\nu} t^{\nu}\right) \qquad \qquad n\in\mathbb{N},\ t\geq0 
 \label{Distribution of exc Poisson as n-th derivative of ML()} 
    \end{align}
    where we used (\ref{lamperti laplace}). For $n = 0$ we immediately get the thesis. For $n\geq 1$ we can now use formulae \eqref{Formula di Faa di Bruno per la derivazione di funzioni composte} and \eqref{MLDerivative}, to get
    \begin{align*}
         p_n(t) &= \frac{t^n}{n!}(-1)^n n! \sum_{k =1}^n \mathcal{M}_{\nu, k\nu + 1}^{k+1} (-\lambda^{\nu}t^{\nu}) \sum_{h_1 +\cdots + h_k = n} (-1)^k \lambda^{k\nu} \prod_{s = 1}^k\frac{(t^{\nu})^{(h_s)}}{h_s!}
         \\ &= (-1)^n t^n \sum_{k =1}^n  \mathcal{M}_{\nu, k\nu + 1}^{k+1} (-\lambda^{\nu}t^{\nu}) t^{k\nu - n}\sum_{h_1 +\cdots + h_k = n} (-1)^k \lambda^{k\nu} \prod_{s = 1}^k\frac{(\nu)_{h_s}}{h_s!}
         \\ &= \sum_{k = 1}^n (-1)^{(n+k)}  (\lambda t)^{k \nu} \mathcal{M}_{\nu, k\nu + 1}^{k+1}(-\lambda^{\nu}t^{\nu}) c(k,n;\nu).
    \end{align*}
    \\
\end{proof}

\begin{os}
An alternative proof of the previous Theorem, also based on formulae \eqref{Formula di Faa di Bruno per la derivazione di funzioni composte} and \eqref{MLDerivative}, is now proposed. Firstly, starting from Equation \eqref{moment generating function of Q(Lt)}, we have that the probability generating function of $N_t$ is given by
\begin{align}
    \label{PGF}
    \mathbb{E}\left[u^{N_t}\right] &= \mathcal{M}_{\nu}\l -\lambda^{\nu}t^{\nu} (1-u)^{\nu} \r, \qquad |u|\leq 1. 
\end{align}
Hence
\begin{align}
    p_n(t) = \frac{1}{n!}\l \frac{d}{dt} \r^n \mathcal{M}_{\nu}(-\lambda^{\nu}t^{\nu} (1-u)^{\nu}) \bigg|_{u = 0}\label{Distribution of N(t) con p_n(t) ricavata dalla funzione generatrice delle probs}
\end{align}
which, by applying \eqref{Formula di Faa di Bruno per la derivazione di funzioni composte} and \eqref{MLDerivative}, coincides with Equation \eqref{Distribution of N(t) formule di Faa di Bruno}.
\end{os}

\subsection{Simulations}\ \\
This paragraph presents some results concerning the numerical and Monte Carlo simulations of the analytical formulas derived in the previous section. We have used the software R and the libraries MittagLeffleR, kStatistics and stabledist. Specifically, Figure \ref{Comparison between theoretical and simulated values for the exchangeable fractiona Poisson} shows the values of \( p_n(t) \) for \( n \in \{0, \ldots, 9\} \) of the exchangeable fractional Poisson process,  obtained by using the analytical formula \eqref{Distribution of N(t) formule di Faa di Bruno} (red dots) and the corresponding simulated values (green triangles), in the case $\lambda =1$. Notably, these results are essentially coinciding. For the Monte Carlo simulation, we have used \(1,000,000 \) independent random numbers following the Lamperti distribution; to this aim, we used that a Lamperti random variable of parameter $\nu$ is equal in law to the ratio of two independent, positive, $\nu$-stable random variables (see \cite{Lamperti_James} for details). Then we have generated the vector ${\rm N}$, of length $1,000,000$, whose $j$-th component is a realization of a Poisson random variable of parameter $ t {\rm L} [j]$, being ${\rm L}$ the vector which contains the realizations the Lamperti variable. Finally we computed the relative frequency of the events that approximates \( p_n(t) \) for each \( n \). For the computation of the analytical values, we have used the expression of \eqref{Formula di Faa di Bruno per la derivazione di funzioni composte} in terms of exponential Bell polynomials (see \cite{comtet2012advanced} for details). Specifically, used the formula \eqref{Distribution of exc Poisson as n-th derivative of ML()}, firstly computing a function of $n$ and $t$ which gives the value of the coefficient and then multiplying it by the $n$-th time derivative of the Mittag-Leffler with parameter $-\lambda^{\nu}t^{\nu}$. We repeated the simulation for $t = 1,2,5$ and $\nu = 0.1, 0.5, 0.9$.

\begin{figure}
    \centering
    \includegraphics[width = 17cm]{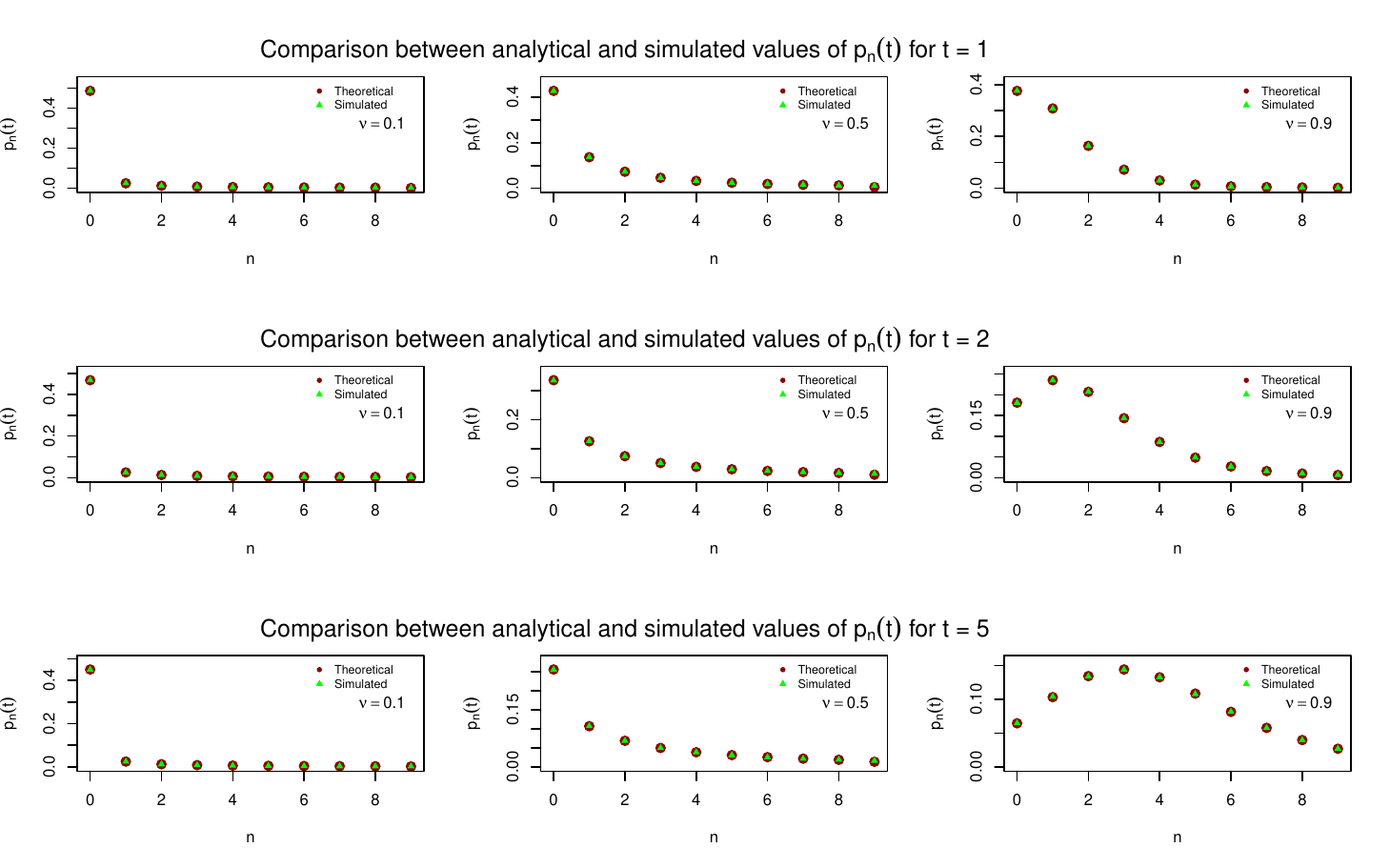}
    \caption{Comparison between analytical values of $p_n(t)$ obtained with formula \eqref{Distribution of N(t) formule di Faa di Bruno} and corresponding simulated values, for $t = 1,2,5$ and $\nu = 0.1, 0.5, 0.9$.}
    \label{Comparison between theoretical and simulated values for the exchangeable fractiona Poisson}
\end{figure}

Figure \ref{Trajectory of exchangeable fractional for different values of nu} shows a trajectory of the exchangeable fractional Poisson process defined in Definition \ref{Definition of exchangeable fractional Poisson process} for different values of the parameter $\nu$, up to $n = 10,000$ events. Here we have generated $n$ values from a Lamperti distribution of parameter $\nu$, for $\nu = 0.5, 0.75, 0.9$, as explained before, and we have saved them in a vector L. Then we have generated the first $n$ waiting times, each of them as a realization of an exponential distribution with parameter L$[j]$, according to Theorem \ref{Theorem: time-change for exchangeable fractional Poisson process}. Finally, the trajectory is obtained cumulating the waiting times. The waiting times become longer, as the value of $\nu$ decreases. \\
Note that before performing both simulations, the seed was set to 1. The interested reader can find the code used to generate the figures at \\ {\tt https://github.com/Lorenzo-Facciaroni/Exchangeable-fractional-Poisson}.

\begin{figure}
    \centering
    \includegraphics[width = 17cm, height = 7cm]{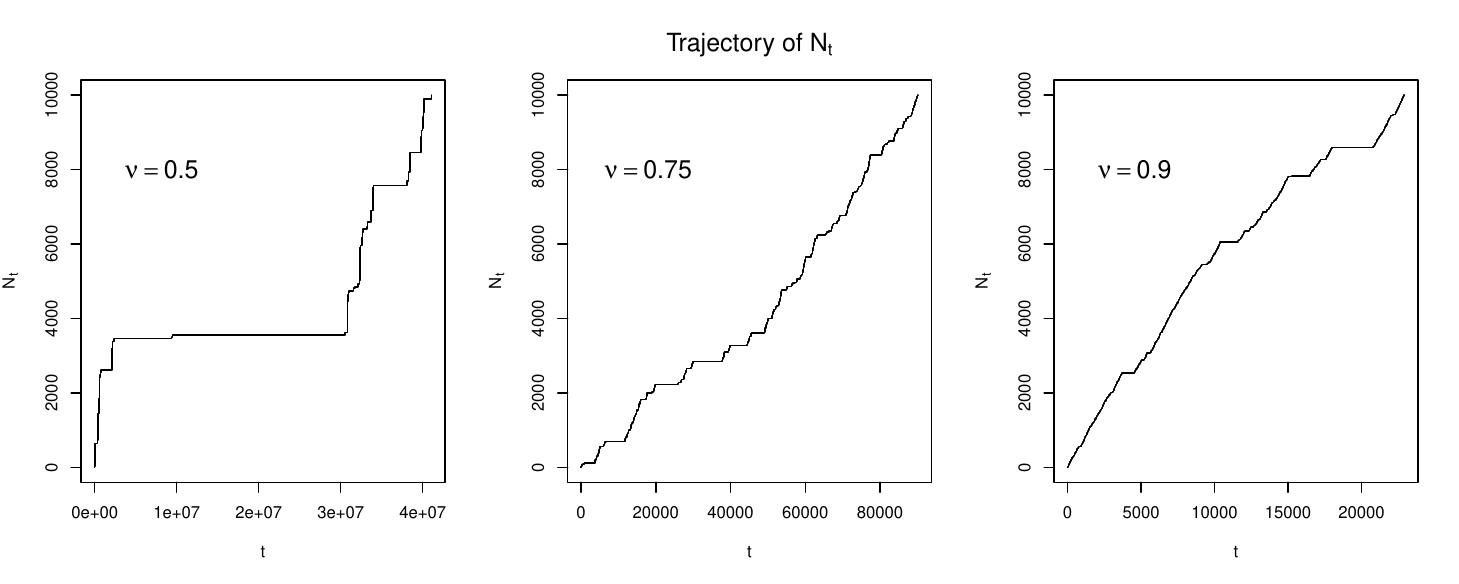}
    \caption{Trajectory of the exchangeable fractional Process $N_t$ for $\nu = 0.5, 0.75, 0.9$ and number of events $n\leq 10000$.}
    \label{Trajectory of exchangeable fractional for different values of nu}
\end{figure}

\newpage
\section{Para-Markov chains in continuous-time}\label{para-Markov section}

We here introduce  para-Markov chains, which include the exchangeable fractional Poisson process  as a notable case.

\begin{defin} \label{Definition of  para-Markov chains}
Let $Y = \{Y_n,\ n\in \mathbb{N}\}$ be a discrete time Markov chain on a finite or countable state space $\mathcal{S}$. For $\nu \in (0,1]$ and   $\lambda : \mathcal{S} \to (0,\infty)$, let $\{J_k\}_{k=1}^\infty$ be a sequence of non-negative random variables, such that,  $\forall n\in \mathbb{N}\setminus \{0\}$,
\begin{align} \label{conditional dependence structure for para-Markov}
P(J_1>t_1, \dots, J_n>t_n |Y_0=y_0, \dots, Y_{n-1}=y_{n-1})= \mathcal{M}_\nu \left ( -\left (\sum _{k=1}^{n} \lambda (y_{k-1})t_{k} \right )^\nu \right ) 
\end{align}
 where $t_k\geq0,\ k \in\{1,\ldots,n\}$. 
A continuous-time chain $X = \{X_t,\ t\in [0, \infty)\}$ such that 
\begin{align*}
    X_t = Y_n \qquad \qquad t\in[T_n, T_{n+1}),\ n\in\mathbb{N}
\end{align*}
where $T_n:= \sum_{k=1}^nJ_k$ and $T_0 := 0$, is said to be a continuous-time para-Markov chain.
\end{defin}

Note that if $\nu =1$ one re-obtains the joint survival function \eqref{Preliminaries: conditional dependence structure of continuous-time Markov chains} and then the process is a continuous-time Markov chain.  
For $\nu \in (0,1) $  the above process is neither Markov nor semi-Markov, because of the dependence between waiting times $J_k$. 

\begin{os}
    We observe that the Definition \ref{Definition of  para-Markov chains} completely defines the finite dimensional distributions of a para-Markov chain $X$.
 Indeed, denoting by $N_t$ the number of jumps up to time $t$, and letting $t_1 <t_2 <\dots <t_n$ we have
 \begin{align*}
        \mathbb{P}\left [\bigcap _{j=1}^n\{X_{t_j}=x_j\}\right ] &= \sum_{k_1 \leq k_2\leq \cdots \leq k_n}^{\infty}\mathbb{P}\left [\bigcap _{j=1}^n\{X_{t_j}=x_j\}, \, \bigcap _{j=1}^n\{N_{t_j}=k_j\}\right ]\\
        &= \sum_{k_1 \leq k_2\leq \cdots \leq k_n}^{\infty}\mathbb{P}\left [\bigcap _{j=1}^n\{Y_{N_{t_j}}=x_j\}, \, \bigcap _{j=1}^n\{N_{t_j}=k_j\}\right ]\\
        &= \sum_{k_1 \leq k_2\leq \cdots \leq k_n}^{\infty}\mathbb{P}\left [\bigcap _{j=1}^n\{Y_{k_j}=x_j\}, \, \bigcap _{j=1}^n\{T_{k_j} \leq t_j < T_{k_j+1} \}\right ]\\
        &= \sum_{k_1 \leq k_2\leq \cdots \leq k_n}^{\infty}\mathbb{P}\left [  \bigcap _{j=1}^n \left. \{T_{k_j} \leq t_j < T_{k_j+1} \} 
  \, \right| \, \bigcap _{j=1}^n\{Y_{k_j}=x_j\}  \right ]\mathbb{P}\left [\bigcap _{j=1}^n\{Y_{k_j}=x_j\}\right ]
\end{align*}
    where, in principle, the last term can be computed by means of the matrix $H$ of the embedded chain and the waiting time distribution given in Definition \ref{Definition of  para-Markov chains}. 
\end{os}

The reason for using the expression {\em para-Markov} is due to the following Theorem. According to it, $X$ is equal in distribution to a time-changed continuous-time Markov process. The time-change consists in rescaling the time $t$ by a Lamperti random variable.

\begin{te} \label{Proposition:  para-Makov chains as time-change of Markov chains through Lamperti}
Let $M = \{M_t,\ t\in [0,\infty)\}$ be a continuous-time Markov chain defined by \eqref{Preliminaries: conditional dependence structure of continuous-time Markov chains} and  $X = \{X_t,\ t\in [0,\infty)\}$ be a para-Markov chain defined in \ref{Definition of  para-Markov chains}.
Let $L$ be a Lamperti random variable, as defined in  \ref{Definition of Lamperti distribution}.
Then we have
    \begin{align*}
        X_t \overset{d}{=} M_{Lt} \qquad \qquad \forall t \geq0,
    \end{align*}
    where $\overset{d}{=}$ denotes equality of  finite dimensional distributions.
\end{te}

\begin{proof}

Let $\{\theta _k\} _{k=1}^\infty$ be a sequence of exponential random variables as in
\eqref{Preliminaries: conditional dependence structure of continuous-time Markov chains}.
Then we have
$$M_t=Y_n,  \qquad   \tau _n \leq t< \tau _{n+1},    $$
where $\tau _n := \sum _{k=1}^n \theta _k$ and $\tau _0:=0$. By the random scaling of time $t \to Lt$ we have
$$M_{Lt}=Y_n, \qquad   \frac{\tau _n}{L} \leq t< \frac{\tau _{n+1}}{L},     $$
which means that the $k$-th waiting time of $M_{Lt}$ is equal to $$ \frac{\theta _k}{L}=\frac{\tau _{k+1}-\tau _k}{L}.$$
Thus, to prove that $M_{Lt}$ coincides with $X_t$ in the sense of finite-dimensional distributions,  it is sufficient to show that the sequence of waiting times  $\{\theta _k/L\} _{k=1}^\infty$ of $M_{Lt}$ has  joint distribution given by \eqref{conditional dependence structure for para-Markov} . This can be done using a conditioning argument, together with the above definition of Lamperti distribution:
\begin{align*}
        \mathbb{P}\biggl [\left. \frac{\theta _1}{L} > t_1,\ \ldots,\ \frac{\theta _n}{L} > t_n \right|  Y_0=y_0, \ldots Y_{n-1}=y_{n-1} \biggr ] &=
        \int_0^{\infty}e^{-l \sum_{k=1}^n \lambda (y_{k-1})t_k}\mathbb{P} (L\in dl)
        \\ &= \mathbb{E}\left[e^{- \left(\sum_{k=1}^n \lambda (y_{k-1}) t_k\right) L }\right]
        \\ &= \mathcal{M}_{\nu}\left(- \left(\sum_{k=1}^n \lambda (y_{k-1})t_k\right)^{\nu}\right)
    \end{align*}

    This completes the proof.\\
\end{proof}

As explained before, the Definition of  para-Markov process \ref{Definition of  para-Markov chains} holds even with  $\mathcal{S}$ a countable set. However,
from now on, we shall consider the case of finite state space, say $|\mathcal{S}| = n \in\mathbb{N}$.  Without loss of generality, we shall write $\mathcal{S} =\{ 1,\ldots, n \}$. In this scenario, the generator $G$ of the Markov process $M$ in Theorem \ref{Proposition:  para-Makov chains as time-change of Markov chains through Lamperti} is an $n \times n$ matrix. Moreover, the systems of Kolmogorov backward and forward equations \eqref{kolmogorov equations markov} have the following solution
\begin{align*}
    P(t) = e^{Gt}.
\end{align*}
Furthermore, from Equation \eqref{generatore markoviano component-wise} the following decomposition holds in matrix from  
\begin{align*}
    G = \Lambda(H-I)
\end{align*}
being $\Lambda = \text{diag}\l \lambda(1),\ldots, \lambda(n)\r$ and $I$ the identity matrix. \\

\begin{os} The above considerations allow us to reinterpret Theorem \ref{Proposition:  para-Makov chains as time-change of Markov chains through Lamperti}
as follows.
The transition matrix of the Markov chain $M$  can be written as $P(t) = e^{\Lambda(H-I) t}$.  Then changing time $t \to Lt$ is equivalent to replacing $\Lambda$ with $L \Lambda$, i.e. rescaling the time parameter is equivalent to rescaling the expectation of each waiting time. 
\end{os}

The next Theorem is the main result of the paper and gives us the governing equation of a para-Markov chain as well as its solution, written in matrix form.

For a matrix $A\in\mathbb{C}^{n\times n}$ we shall indicate with $\rho(A)$ the spectral radius of $A$ and $\sigma(A)$ the spectrum. 
 We use the natural norm $||\cdot||: \mathbb{C}^{n\times n}\rightarrow \mathbb{R}^{+}$; this is a matrix norm induced by a vector norm, i.e. 
\begin{align*}
    ||A|| := \sup_{||x||_{v} = 1} ||Ax||_{v},
\end{align*}
where $||\cdot||_{v}:\mathbb{C}^n\rightarrow\mathbb{R}^{+}$ is a vector norm. Moreover,
     $A$ is said to be convergent if there exists a natural norm such that
    \begin{align*}
        \lim_{k\to\infty} ||A^k|| = 0.
    \end{align*}
 We shall use the notation $\mathbb{C}^{-} := \{z\in\mathbb{C}\ s.t.\ \Re\{z\} < 0\}$.
 
 For a scalar function $f:\mathbb{C}\rightarrow\mathbb{C}$, we refer to the meaning of $f(A)$, being $A\in\mathbb{C}^{n\times n}$, as discussed in Chapter 1 of \cite{higham2008functions}. Specifically, let $A$ have canonical Jordan decomposition $A = Z^{-1}JZ$, where $J$ is the block diagonal matrix, while $Z$ is the matrix whose columns contain the generalized eigenvectors. Hence $J = \mathrm{diag}({J_{m_1}(\alpha_1), \ldots, J_{m_p}(\alpha_p}))$, where $J_{m_k}(\alpha_k)$ denotes a Jordan block with dimension $m_k$ corresponding to the eigenvalue $\alpha_k$, i.e. it has $\alpha _k$ on the diagonal and $1$ above the diagonal; eigenvalues related to distinct blocks do not need to be distinct.
For $f(A)$ to be well defined, we need to require that $f(\cdot)$, as a scalar function, is defined on the spectrum of $A$, i.e. there must exist the derivatives
\begin{align*}
     f^{(j)}(\alpha_k), \ \ \ j = 0,\ldots, n_k - 1,\ k = 1,\ldots, s
\end{align*}
     with $s$ the number of distinct eigenvalues of $A$ and $n_k$ the order of the largest Jordan block where $\alpha_k$ appears. We say that $n_k$ is the \textit{index} of $\alpha _k$. In this case we can use the Jordan canonical decomposition $A = Z^{-1}JZ$ to define 
     \begin{align}
        f(A) := Z^{-1} f(J) Z = Z^{-1} \text{diag}\left(f(J_{m_1}(\alpha_1)), \ldots, f(J_{m_p}(\alpha_p))\right) Z  \label{matrice a blocchi}
    \end{align}
    being $p$  the number of Jordan blocks, i.e. the number of independent eigenvectors of $A$, and
    \begin{align*}
        f\left(J_{m_i} (\alpha _i)\right) := 
        \left[\begin{matrix}
            f(\alpha_i) & f'(\alpha_i) & \cdots & \frac{f^{(m_i-1)}(\alpha_i)}{(m_i - 1)!}
            \\ & f(\alpha_i) & \ddots & \vdots
            \\ & & \ddots & f'(\alpha_i)
            \\ & & & f(\alpha_i)
        \end{matrix}\right].
    \end{align*}

\begin{os}
Since the Mittag-Leffler function $\mathcal{M}_\nu$ defined in \eqref{funzione Mittag Leffler} is entire, then it is defined on the spectrum of any matrix $A \in \mathbb{C}^{n\times n}$.
 Let $A$ have Jordan decomposition $A= Z^{-1}JZ$. Then, being $M_\nu$ defined by a power series,  the matrix $\mathcal{M}_\nu (A)$ can be explicitly obtained as follows (see \cite{garrappa2017})
\begin{align}
    \mathcal{M}_\nu (A) &= \sum _{k=0}^\infty \frac{A^k}{\Gamma (1+\nu k)} =    \sum _{k=0}^\infty \frac{\overbrace{Z^{-1}JZ\ Z^{-1}JZ \cdots\ Z^{-1}JZ}^{\text{k times}}}{\Gamma (1+\nu k)} \nonumber
    \\ &=        Z^{-1} \biggl ( \sum _{k=0}^\infty \frac{J^k}{\Gamma (1+\nu k)}\biggr ) Z = Z^{-1}\mathcal{M}_\nu (J)Z
\end{align}
 which coincides with expression given in \eqref{matrice a blocchi}.
     For $\nu = 1$, we have $\mathcal{M}_{\nu}(x)= e^x$ and thus we get the exponential of a matrix
    as \begin{align*}
        e^A = Z^{-1} e^J Z.
    \end{align*}
   In this case, to compute $e^J$ explicitly, we observe that each Jordan block can be decomposed as 
    \begin{align*}
        J_{m_k} = \alpha_k I + N_k
    \end{align*}
    where $I$ is the identity matrix and $N_k$ is nilpotent of order $m_k$. The matrices $\alpha_k I$ and $N_k$ commute and thus the $k$-th block is given by
    \begin{align*}
        e^{J_{m_k}} = e^{\alpha_k}e^{N_k} = e^{\alpha_k} \sum_{s=0}^{m_k-1} \frac{(N_k)^s}{s!},
    \end{align*}
i.e. it is  sufficient to compute a finite sum.
\end{os}

Before stating the Theorem, we recall that a matrix $A$ is said to be \textit{irreducible} if it is not similar via a permutation to a block upper triangular matrix, i.e. it does not have invariant subspaces. Indeed, if the generator of a continuous time Markov chain is irreducible, then there is a non-zero probability of transitioning from any state to any other state.

\begin{te} \label{Theorem: G irriducibile allora abbiamo 3 tesi. etc etc}

 Let us consider a para-Markov chain $X$ and the related Markov chain $M$, in the sense of Theorem \ref{Proposition:  para-Makov chains as time-change of Markov chains through Lamperti}, with generator $G$. Let  $P(t)=[p_{ij}(t)]$ be the transition matrix of $X$, i.e. $p_{ij}(t)=\mathbb{P}(X_t=j|X_0=i),\ i,j\in\mathcal{S}$.
    If $G$ is irreducible, then 
    \begin{enumerate}
        \item The matrix $-(-G)^\nu$ exists for any $\nu\in(0,1]$,
        \item The transition matrix has the form $$P(t) = \mathcal{M}_{\nu}(-(-G)^{\nu}t^{\nu}),$$
        \item $P(t)$ is the solution of
        \begin{align}
            \frac{d^{\nu}}{dt^{\nu}} P(t) = -(-G)^{\nu} P(t) \label{equazione principale}
        \end{align}
        with initial condition $P(0) = I$.
    \end{enumerate}
\end{te}
\begin{proof}
    Let us split the proof in three parts.
    \begin{enumerate}
        \item  Since the function $f(x) = (-x)^{\nu},\ \nu\in(0,1]$, is not differentiable at $x=0$, then, according to \eqref{matrice a blocchi}, it  is defined on the spectrum of $G$ if either $G$ does not have the eigenvalue $0$ or $G$ does have the eigenvalue $0$ with index $n=1$. However, we shall see that $G$ necessarily has the eigenvalue $0$.
        Thus, we shall show that a sufficient condition for $G$ to have the eigenvalue $0$ with index $1$ is its irreducibility; indeed,  irreducibility of $G$ implies that $0$ is a simple eigenvalue, i.e. its algebraic multiplicity is 1.
        
        We indicate with $\mathbf{1}$ the vector in $\mathbb{R}^{n}$ with all coordinates equal to 1. The row sums are 0 by  (\ref{generatore markoviano component-wise}), which gives 
    \begin{align*}
        G \mathbf{1} = \mathbf{0}.
    \end{align*}
    It means that $G$ has $0$ as eigenvalue with correspondent eigenvector $\mathbf{1}$.
    Moreover, we know that $g_{ij} \geq 0,\ i\neq j $ so, given the definition of the diagonal elements, we define
    \begin{align*}
        R_{ii} := \sum_{j\neq i} |g_{ij}| = \sum_{j\neq i} g_{ij} = -g_{ii} \ \ \ i\in\{1,\ldots, n\}
    \end{align*}
    which implies that the so-called Gershgorin discs $D(g_{ii}, R_{ii})= \{ z\in \mathbb{C}\ s.t. \ |z-g_{ii}|\leq R_{ii} \}$ are subsets of $\mathbb{C}^{-} \cup \{0\}$.
    The Gershgorin theorem in \cite{gershgorin1931uber} ensures that all the eigenvalues of $G$ lie in the union of such discs, which, in our case,  is contained   in    $\mathbb{C}^{-} \cup \{0\}$.

    Now, letting $\eta := \max\{g_{ij},\ i,j\in \mathcal{S}\}$ and considering that    $\rho(G) > 0$, the matrix $T$ defined by
    \begin{align*}
        T := \frac{1}{\eta\rho(G)} G + I
    \end{align*}
    is irreducible as well and it has non-negative entries. It follows by linearity of the eigenvalues that $\rho(T) = 1$ is an eigenvalue of $T$. Moreover, since the eigenvalues of $G$ lie in   $\mathbb{C}^{-} \cup \{0\}$, then all the
     eigenvalues of $T$ lie in $D\left( \frac{1}{2}, \frac{1}{2} \right)$, which is the closed disc centered in $\frac{1}{2}+0i$ and radius $\frac{1}{2}$. The Perron-Frobenius theorem (see Paragraph 8.3 in \cite{horn2012matrix}) guarantees that $1$ is actually a simple eigenvalue and it is called the Perron-Frobenius eigenvalue. By applying the inverse formula $T\mapsto \eta \rho(G)(T-I)$, we get that $\alpha = 0$ is simple for $G$. \\

     \item 
     For a Lamperti random variable $L$, the function 
     \begin{align}z \mapsto \int _{0}^\infty e^{ztl} \mathbb{P}(L\in dl) \label{vvvv}\end{align}
     is well defined for $\Re\{z\}\leq 0$. Moreover (\ref{vvvv}) is analytic for $z\neq 0$, which is clear also by expressing it as
\begin{align}z \mapsto \int_0^{\infty}e^{ztl}\mathbb{P}(L\in dl) = \mathcal{M}_\nu (-t^\nu (-z)^\nu). \end{align}
      Such a function is well defined on the spectrum of $G$, since $G$ has $0$ as a simple eigenvalue by the irreducibility assumption, and furthermore all the other eigenvalues have negative real part. By virtue of this consideration, together with the time-change Theorem \ref{Proposition:  para-Makov chains as time-change of Markov chains through Lamperti}, $P(t)$ takes the following  matrix form in the sense of \eqref{matrice a blocchi}
\begin{align*}
P(t) &= \int_0^{\infty}e^{Gtl}\mathbb{P}(L\in dl)= \mathcal{M}_{\nu}(-(-G)^{\nu}t^{\nu}),
\end{align*}
 as desired.

     \item Assume, for the moment, that $\frac{d^\nu}{dt^\nu}P(t)$    exists, it is continuous and Laplace transformable. To prove the statement, we preliminary recall that a square matrix $B$ is convergent iff $\rho(B) < 1$; in this case $I-B$ is non-singular, such that 
    \begin{align}
        (I - B)^{-1} = \sum_{k = 0}^{\infty} B^k. \label{serie geometrica matrici}
    \end{align}
Now, let us consider $g(t) = \mathcal{M}_{\nu}\left(-A t^{\nu}\right)$, being $A\in\mathbb{C}^{n\times n}$, and compute the Laplace transform
 \begin{align}
\Tilde{g}\left(s\right) &= \int_0^{\infty}e^{-st} \mathcal{M}_{\nu}\left(-A t^{\nu}\right) dt,  \qquad s\in \mathbb{C},
\end{align}
where the integral is meant component-wise. Being $\mathcal{M}_\nu$ entire, we have     
    \begin{align*}
        \Tilde{g}\left(s\right) &= \int_0^{\infty}e^{-st}\sum_{k = 0}^{\infty} \frac{(-1)^kA^k t^{k\nu}}{\Gamma(1 + \nu k)}dt  
        \\ &= \frac{1}{s}\sum_{k = 0}^{\infty}\frac{(-1)^kA^k }{s^{\nu k}}  
        \\ &= \frac{1}{s}\sum_{k = 0}^{\infty}\left(\frac{-A }{s^{\nu}}\right)^k  
    \end{align*}
and then $\tilde{g}(s)$ converges for all $s$ such that the spectral radius of $-A/s^\nu$ is less than $1$, namely the Laplace transform is certainly defined for $\Re\{s\}>(\rho (A)) ^{1/\nu}$. Moreover, the Laplace inversion Theorem  ensures that $\tilde{g}(s)$ is analytic in the same region  $\Re\{s\}>(\rho (A)) ^{1/\nu}$ and thus uniquely identifies $g(t)$. Using \eqref{serie geometrica matrici} we obtain 
    \begin{align}
        \Tilde{g}\left(s\right) &= \frac{1}{s}\left(I + \frac{A}{s^{\nu}}\right)^{-1} \nonumber
        \\ &= s^{\nu-1}\left(s^{\nu} I + A\right)^{-1} \ \ \ \ \ \ \  \Re\{s\} > \left(\rho(A)\right)^{\frac{1}{\nu}}. \label{Laplace transform of Mittag-Leffler function WITH MATRIX ARGUMENT}
    \end{align}

    We now look at the solution of the following problem
    \begin{align}
        \frac{d^{\nu}}{dt^{\nu}} h(t) = - A h(t) \qquad  h(0) = I,   \label{Equation to solve to be the eigenfunction with matricx argument of the fractional caputo derivative}
    \end{align}
   being $A\in\mathbb{C}^{n\times n}$. By applying the Laplace transform component-wise on both sides we get
    \begin{align*}
        s^{\nu} \Tilde{h}\left(s\right) - s^{\nu - 1} h(0) &= - A \Tilde{h}\left(s\right) \qquad s\in \mathbb{C}
       \end{align*}
     namely  
       \begin{align*}
        \left(I + \frac{A}{s^{\nu}}\right)\Tilde{h}\left(s\right) &= s^{-1} I \qquad s\in \mathbb{C}
    \end{align*}
    For $\Re\{s\} > \left(\rho(A)\right)^{\frac{1}{\nu}}$, we have that  $\left(I + \frac{A}{s^{\nu}}\right)$ is non-singular and then
    \begin{align*}
        \Tilde{h}\left(s\right) = s^{\nu-1}\left(s^{\nu} I + A\right)^{-1} \ \ \ \ \ \ \ \Re\{s\} > \left(\rho(A)\right)^{\frac{1}{\nu}} 
    \end{align*}
    which coincides with the Laplace transform \eqref{Laplace transform of Mittag-Leffler function WITH MATRIX ARGUMENT}. The inverse Laplace transform ensures equality for almost all $t>0$; moreover, continuity of $t\to P(t)$, which stems from the expression in point $(2)$ of the present Theorem, ensures equality for all $t>0$. Hence $h(t) = \mathcal{M}_{\nu}\left(-A t^{\nu}\right)$ solves the problem
   \eqref{Equation to solve to be the eigenfunction with matricx argument of the fractional caputo derivative}.    To conclude, we finally  set  $A = (-G)^{\nu}$.
   
   It remains to prove that $\frac{d^\nu}{dt^\nu}P(t)$    exists and it is continuous. The convolution $t\to \int _0 ^t \bigl (P(\tau)-P(0) \bigr ) \frac{(t-\tau) ^{-\nu}}{\Gamma (1-\nu)} d\tau$ is well defined (see Prop. 1.6.4 in \cite{arendt}).    By using similar calculations as above (of which we omit the details),
   it is easy to prove that the two functions $t\to \int _0 ^t \bigl (P(\tau)-P(0) \bigr ) \frac{(t-\tau) ^{-\nu}}{\Gamma (1-\nu)} d\tau$ and $t\to \int _0 ^t-(-G)^\nu P(\tau)d\tau$ have the same Laplace transform. Hence they coincide for almost all $t>0$. Moreover, both functions are  continuous since $P(t)$ is continuous by the expression given in point $(2)$ of the Theorem. Hence the two functions coincide for any $t>0$: 
\begin{align}   
   \int _0 ^t \bigl (P(\tau)-P(0) \bigr ) \frac{(t-\tau) ^{-\nu}}{\Gamma (1-\nu)} d\tau = \int _0^t -(-G)^\nu P(\tau) d \tau. \label{uuu}
   \end{align}
   The right side of (\ref{uuu}) is differentiable for $t>0$ because $-(-G)^{\nu}P(\tau)$ is component-wise continuous as it is a linear combination of continuous functions, and this is true also for the left side because the equality holds pointwise. Hence the Caputo derivative exists and is continuous. Now if we apply the time derivative to both sides, we obtain the desired equation.
    \end{enumerate}
    \ \\
\end{proof}
\begin{os}
By the above considerations, the matrix $(-G)^\nu$ is given by
     \begin{align*}
         (-G)^{\nu} = Z^{-1} \text{diag}\left((J_{m_1}(\alpha_1))^{\nu}, \ldots, (J_{m_p}(\alpha_p))^{\nu}\right) Z
     \end{align*}
     where 
     \begin{align*}
        \left(J_{m_i}(\alpha_i)\right)^{\nu} = 
        \left[\begin{matrix}
            \alpha_i^{\nu} & \nu\alpha_i^{\nu-1} & \cdots & \frac{(\nu)_{m_i-1}\alpha _i ^{\nu -m_i+1}}{(m_i-1)!}
            \\ & \alpha_i^{\nu} & \ddots & \vdots
            \\ & & \ddots & \nu\alpha_i^{\nu-1}
            \\ & & & \alpha_i^{\nu}
        \end{matrix}\right]
    \end{align*}
   and one can see that the eigenvalue $0$ must have index $1$, being $z\mapsto (-z)^\nu$  not differentiable at $0$ for $\nu \in (0,1)$. 
\end{os}

\begin{os}
Equation (\ref{equazione principale}) does not uniquely identify our para-Markov chain. For example, consider the process  $M(H(L(t)))$ where $M$ is a Markov chain with generator $G$, $H$ is a stable subordinator with index $\nu$ and $L$ is an inverse stable subordinator with index $\nu$, under the assumption that $M$, $H$ and $L$ are independent. This process is governed by the same equation (\ref{equazione principale}), even if it is not a para-Markov chain but a semi-Markov one. 
\end{os}

\section{Final remarks}
By using the same techniques as in the previous section, we find an interesting result on semi-Markov chains. Consider, indeed, semi-Markov chains with Mittag-Leffler waiting times recalled in Section \ref{preliminaries}, i.e. those governed by the equations \eqref{Kolmogorov backward and forward}.
\begin{prop} \label{distribuzione semi Markov}
    If the state space $\mathcal{S}$ is finite, then the solution of \eqref{Kolmogorov backward and forward} has the following matrix form
     \begin{align*}
         P(t) = \mathcal{M}_{\nu}(G t^{\nu}) .
     \end{align*}
\end{prop}
\begin{proof}
    It is sufficient to adapt the arguments used in the proof of point $3)$ of theorem \ref{Theorem: G irriducibile allora abbiamo 3 tesi. etc etc}, setting $A = -G$.
\end{proof}
To the best of our knowledge, the result in Proposition \ref{distribuzione semi Markov} is new. Indeed, in the literature,  $P(t)$ has been written by using the composition of the corresponding Markov process with an inverse stable subordinator (see e.g. \cite{toaldo2019}), but the explicit solution in matrix form has never been written.

Table \ref{Table: comparison between continuous-time Markov, semi-Markov and  para-Markov processes} sums up the main facts we have discussed on Markov, semi-Markov and para-Markov chains. 

\begin{table}[!ht]
    \centering
    \begin{tabular}{|c|c|c|} \hline 
         &  Governing Equation& Solution\\ \hline 
         Markov & $\frac{d}{dt}P(t) = G P(t)$ &  $P(t) = e^{Gt}$ \\ \hline 
         semi-Markov & $\frac{d^{\nu}}{dt^{\nu}}P(t) = GP(t)$ & $P(t) = \mathcal{M}_{\nu}\left(Gt^{\nu}\right)$ \\ \hline
  para-Markov& $\frac{d^{\nu}}{dt^{\nu}} P(t) = -(-G)^{\nu} P(t)$ & $P(t) = \mathcal{M}_{\nu}(-(-G)^{\nu}t^{\nu})$ \\\hline
    \end{tabular}
    \caption{Comparison between continuous-time Markov, semi-Markov and  para-Markov processes.}
    \label{Table: comparison between continuous-time Markov, semi-Markov and  para-Markov processes}
\end{table}
Note that for $\nu =1$ semi-Markov and para-Markov chains reduce to Markov ones.
Once again, we stress the fact that the governing equation of Markov chains is driven by the first derivative, which is a local operator, whereas the governing equations of the semi-Markov and  para-Markov chains depend on the Caputo derivative of order $\nu$, which is a non-local operator. It is related to the characteristic of the processes themselves: the probability of a future state depends both on the present value of the process and also on the past. \\

\vspace{0.5cm}

\section*{Acknowledgements}
The authors acknowledge financial support under the National
Recovery and Resilience Plan (NRRP), Mission 4, Component 2, Investment 1.1, Call for tender No.
104 published on 2.2.2022 by the Italian Ministry of University and Research (MUR), funded by the
European Union - NextGenerationEU- Project Title “Non-Markovian Dynamics and Non-local Equations”
- 202277N5H9 - CUP: D53D23005670006 - Grant Assignment Decree No. 973 adopted on June 30, 2023,
by the Italian Ministry of Ministry of University and Research (MUR). \\

The author Bruno Toaldo would like to thank the Isaac Newton Institute for Mathematical Sciences, Cambridge, for support and hospitality during the programme Stochastic Systems for Anomalous Diffusion, where work on this paper was undertaken. This work was supported by EPSRC grant EP/Z000580/1.

\section*{Conflict of interests}
The authors declare there is no conflict of interests.

\section*{Data statement}
Data and programs concerning this paper can be found at \\ {\tt https://github.com/Lorenzo-Facciaroni/Exchangeable-fractional-Poisson}.

\bibliographystyle{plain}
\bibliography{Bibliografia}

\begin{thebibliography}{10}

\bibitem{arendt}
W.~Arendt, C.~Batty, M.~Hieber, and F.~Neubrander.
\newblock {\em Vector valued Laplace transform and Cauchy problems}.
\newblock Birkhauser, 2010.

\bibitem{ascione2021}
G.~Ascione, N.~Leonenko, and E.~Pirozzi.
\newblock Fractional immigration death processes.
\newblock {\em Journal of Mathematical Analysis and Applications}, (495), 2021.

\bibitem{barlow1992}
R.~E. Barlow and M.~B. Mendel.
\newblock De finetti-type representations for life distributions.
\newblock {\em Journal of the American Statistical Association}, 87(420):1116--1122, 1992.

\bibitem{orsingherbeghin2009}
L.~Beghin and E.~Orsingher.
\newblock Fractional poisson processes and related planar random motions.
\newblock {\em Electronic Journal of Probability}, 14(61):1790--1826, 2009.

\bibitem{beghin2019}
L.~Beghin and C.~Ricciuti.
\newblock Time-inhomogeneous fractional poisson processes defined by the multistable subordinator.
\newblock {\em Stochastic Analysis and Applications}, 37(2):171--188, 2019.

\bibitem{caramellino1994dependence}
L.~Caramellino and F.~Spizzichino.
\newblock Dependence and aging properties of lifetimes with schur-constant survival functions.
\newblock {\em Probability in the Engineering and Informational Sciences}, 8(1):103--111, 1994.

\bibitem{caramellino1996wbf}
L.~Caramellino and F.~Spizzichino.
\newblock Wbf property and stochastical monotonicity of the markov process associated to schur-constant survivial functions.
\newblock {\em Journal of Multivariate Analysis}, 56(1):153--163, 1996.

\bibitem{comtet2012advanced}
L.~Comtet.
\newblock {\em Advanced Combinatorics: The art of finite and infinite expansions}.
\newblock Springer Science \& Business Media, 2012.

\bibitem{degregorio2021}
A.~De~Gregorio and F.~Iafrate.
\newblock Telegraph random evolutions on a circle.
\newblock {\em Stochastic Processes and their Applications}, (141):79--108, 2021.

\bibitem{dicrescenzo2016}
A.~Di~Crescenzo, B.~Martinucci, and A.~Meoli.
\newblock A fractional counting process and its connection with the poisson process.
\newblock {\em ALEA}, (13):291--307, 2016.

\bibitem{faa1855}
F.~Faà~di Bruno.
\newblock Sullo sviluppo delle funzioni.
\newblock {\em Annali di Scienze Matematiche e Fisiche}, 6:479--868, 1855.

\bibitem{garra2015}
R.~Garra, E.~Orsingher, and F.~Polito.
\newblock State-dependent fractional point processes.
\newblock {\em Journal of applied Probability}, 52(1):18--36, 2015.

\bibitem{garrappa2017}
R.~Garrappa and M.~Popolizio.
\newblock Computing the matrix {M}ittag-{L}effler function with applications to fractional calculus.
\newblock {\em Journal of Scientific Computation}, 77:129--153, 2017.

\bibitem{georgiu2015}
N.~Georgiou, I.~Z. Kiss, and E.~Scalas.
\newblock Solvable non-markovian dynamic network.
\newblock {\em Physical Review E}, 92(4):171--188, 2015.

\bibitem{gershgorin1931uber}
S.~A. Gershgorin.
\newblock Uber die abgrenzung der eigenwerte einer matrix.
\newblock {\em News of the Russian Academy of Sciences. Mathematical series}, (6):749--754, 1931.

\bibitem{gupta2023}
N.~Gupta and A.~Kumar.
\newblock Fractional poisson processes of order k and beyond.
\newblock {\em Journal of theoretical probability}, 36(4):2165--2191, 2023.

\bibitem{higham2008functions}
N.~J. Higham.
\newblock Functions of matrices: Theory and computation, 2008.

\bibitem{horn2012matrix}
R.~Horn and C.~R. Johnson.
\newblock {\em Matrix analysis}.
\newblock Cambridge university press, 2012.

\bibitem{Lamperti_James}
L.~F. James.
\newblock {Lamperti-type laws}.
\newblock {\em The Annals of Applied Probability}, 20(4):1303 -- 1340, 2010.

\bibitem{kataria2019}
K.~K. Kataria and P~Vellaisamy.
\newblock On distributions of certain state-dependent fractional point processes.
\newblock {\em Journal of theoretical probability}, 32:1554--1580, 2019.

\bibitem{kolokoltsov2009}
V.~Kolokoltsov.
\newblock Generalized continuous-time random walks, subordination by hitting times, and fractional dynamics.
\newblock {\em Theory of Probability \& Its Applications}, 53(4):594--609, 2009.

\bibitem{lamperti1958}
J.~Lamperti.
\newblock An occupation time theorem for a class of stochastic processes.
\newblock {\em Transactions of the American Mathematical Society}, 88:380--387, 1958.

\bibitem{laskin2003}
N.~Laskin.
\newblock Fractional poisson process.
\newblock {\em Communications in nonlinear science and numerical simulation}, 8:201--213, 2003.

\bibitem{leonenko2017}
N.~Leonenko, E.~Scalas, and M.~Trinh.
\newblock The fractional non homogeneous poisson process.
\newblock {\em Statistics and Probability Letters}, 120(4):147--156, 2017.

\bibitem{maheshwari2019}
A.~Maheshwari and P.~Vellaisamy.
\newblock Fractional poisson process time-changed by lévy subordinator and its inverse.
\newblock {\em Journal of Theoretical Probability}, 32(4):1278--1305, 2019.

\bibitem{mainardi2004}
F.~Mainardi, R.~Gorenflo, and E.~Scalas.
\newblock A fractional generalization of the poisson processes.
\newblock {\em Vietnam Journal of Mathematics}, 32(SI):53--64, 2004.

\bibitem{scalas2004}
F.~Mainardi, R.~Gorenflo, and E.~Scalas.
\newblock Uncoupled continuous-time random walks: Solution and limiting behavior of the master equation.
\newblock {\em Physical Review E}, 69.1(1):153--163, 2004.

\bibitem{meerschaert2011}
M.~Meerschaert, E.~Nane, and P.~V.
\newblock The fractional poisson process and the inverse stable subordinator.
\newblock {\em Electronic Journal of Probability}, 129(9):2850--2879, 2019.

\bibitem{toaldo2019}
M.~Meerschaert and B.~Toaldo.
\newblock Relaxation patterns and semi-markov dynamics.
\newblock {\em Stochastic Processes and their Applications}, 129(9):2850--2879, 2019.

\bibitem{metzler2000}
R.~Metzler and J.~Klafter.
\newblock The random walk's guide to anomalous diffusion: a fractional dynamics approach.
\newblock {\em Physics reports}, 339(1):1--77, 2000.

\bibitem{norris1998}
J.~R. Norris.
\newblock {\em Markov chains}.
\newblock Cambridge university press., 1998.

\bibitem{polito2010}
E.~Orsingher and F.~Polito.
\newblock Fractional pure birth processes.
\newblock {\em Bernoulli}, 16(3):858--881, 2010.

\bibitem{orsingher2013}
E.~Orsingher and F.~Polito.
\newblock Randomly stopped nonlinear fractional birth processes.
\newblock {\em Stochastic Analysis and Applications}, 31(2):262--292, 2013.

\bibitem{orsingher2018semi}
E.~Orsingher, C.~Ricciuti, and B.~Toaldo.
\newblock On semi-markov processes and their kolmogorov's integro-differential equations.
\newblock {\em Journal of Functional Analysis}, 275(4):830--868, 2018.

\bibitem{prabhakar1971}
T.~R. Prabhakar.
\newblock A singular integral equation with a generalized mittag-leffler function in the kernel.
\newblock {\em Yokohama Mathematical Journal}, 19(1):7--15, 1971.

\bibitem{ricciuti2017semi}
C.~Ricciuti and B.~Toaldo.
\newblock Semi-markov models and motion in heterogeneous media.
\newblock {\em Journal of Statistical Physics}, 169(2):340--361, 2017.

\bibitem{ricciuti2023semi}
C.~Ricciuti and B.~Toaldo.
\newblock From semi-markov random evolutions to scattering transport and superdiffusion.
\newblock {\em Communications in Mathematical Physics}, 401(3):2999--3042, 2023.

\bibitem{scalas2021}
E.~Scalas and B.~Toaldo.
\newblock Limit theorems for prices of options written on semi-markov processes.
\newblock {\em Theory of Probability and Mathematical Statistics}, 105(3):3--33, 2021.

\end{thebibliography}

\end{document}